\documentclass{article}
\usepackage{amssymb}
\usepackage{amsmath}
\usepackage{amsthm}

\begin{document}

\let\oldphi=\phi
\let\phi=\varphi
\let\oldtheta=\theta
\let\theta=\vartheta

%
\def\cL{{\mathcal L}}
\def\cP{{\mathcal P}}
\def\cH{{\mathcal H}}
\def\cF{{\mathcal F}}
\def\cD{{\mathcal D}}
\def\cS{{\mathcal S}}

\def\ch{\mathop{\rm char}\nolimits}
\def\diag{\mathop{\rm diag}\nolimits}
\def\diam{\mathop{\rm diam}\nolimits}
\def\rank{\mathop{\rm rank}\nolimits}
\def\spn{\mathop{\rm span}\nolimits}
\def\GL{\mathop{\rm GL}\nolimits}
\def\id{\mathop{\rm id}\nolimits}

\def\inv{\,\overline{\phantom{I}}\,}
\newcommand\invstrich{\inv{}^{{}'}}

\newcommand{\MM}{M_{m\times n}(\cD)}
\newcommand{\HH}{\cH_n(\cD)}

\theoremstyle{plain}{
  \newtheorem{theorem}{Theorem}[section]
  \newtheorem{assumption}[theorem]{Assumption}
  \newtheorem{definition}[theorem]{Definition}
  \newtheorem{lemma}[theorem]{Lemma}
  \newtheorem{proposition}[theorem]{Proposition}
  \newtheorem{corollary}[theorem]{Corollary}
  \newtheorem{Preliminaries}[theorem]{Preliminaries}}
\theoremstyle{definition}{
  \newtheorem{example}[theorem]{Example}
  \newtheorem{remark}[theorem]{Remark}}

\renewcommand\labelenumi{\rm\arabic{enumi}.}

\title{Diameter preserving surjections in the geometry of matrices}

\author{Wen-ling Huang\thanks{Lise Meitner Research Fellow
of the Austrian Science Fund (FWF), project M 1023.} \and Hans
Havlicek\thanks{Corresponding author}}


\date{}

\maketitle

\begin{abstract}
We consider a class of graphs subject to certain restrictions, including the
finiteness of diameters. Any surjective mapping $\phi:\Gamma\to\Gamma'$ between
graphs from this class is shown to be an isomorphism provided that the
following holds: Any two points of $\Gamma$ are at a distance equal to the
diameter of $\Gamma$ if, and only if, their images are at a distance equal to
the diameter of $\Gamma'$. This result is then applied to the graphs arising
from the adjacency relations of spaces of rectangular matrices, spaces of
Hermitian matrices, and Grassmann spaces (projective spaces of rectangular
matrices).
\end{abstract}



\section{Introduction}\label{se:intro}

Related to his study of analytic functions of several complex variables,
L.~K.~Hua initiated the geometries of rectangular, symmetric, Hermitian, and
alternate matrices in the middle forties of the last century. The elements of
such a matrix space are also called \emph{points}, and there is a symmetric and
anti-reflexive \emph{adjacency\/} relation on the point set. The adjacency
relation turns the point set of a matrix space into the set of vertices of a
graph. The problem to describe all isomorphisms between such graphs has
attracted many authors. In other words, one aims at describing all bijections
between matrix spaces such that adjacency (graph-theoretic distance $1$) is
preserved in both directions. Solutions to this problem are usually stated as a
\emph{fundamental theorem\/} for a geometry of matrices. See the book of
Z.-X.~Wan \cite{Wan1996a} for a wealth of results and references.
\par
All graphs, which stem from the matrix spaces mentioned above, have finite
diameter. Several recent papers are concerned with a description of all
bijections between matrix spaces which are \emph{diameter preserving\/} in both
directions. The proofs pursue the same pattern: In a first step, a bijection of
this kind is shown to preserve adjacency in both directions. Then, in a second
step, the appropriate fundamental theorem is applied to accomplish the task.
See \cite{HavlicekSemrl2006} and \cite{Kobal2007}. Similar results about
Grassmann spaces and other structures can be found in
\cite{AbramenkoVanMaldeghem2000}, \cite{BlunckHavlicek2005b},
\cite{BlunckHavlicek2005}, and \cite{HavlicekPankov2005}.
\par
In the present paper we aim at shedding light on this issue by a different
approach. It follows the ideas from \cite{Huang2005a}, where adjacency
preserving mappings were exhibited for a wide class of point-line geometries
rather than those of a specific kind. So, we consider a \emph{class of
graphs\/} subject to five conditions (A1)--(A5), one of them ensuring
finiteness of diameters. Theorem~\ref{thm:hauptsatz} contains our main result:
A surjective mapping $\phi$ between graphs $\Gamma$ and $\Gamma'$ from this
class is an isomorphism provided that any two points of $\Gamma$ are at a
distance equal to the diameter of $\Gamma$ if, and only if, their images are at
a distance equal to the diameter of $\Gamma'$. The backbone of the proof is
contained in Lemma~\ref{lem:A1-4}, which is about graphs satisfying (A1)--(A4).
It contains a sufficient condition for two points of such a graph to be
adjacent. This condition is in terms of the diameter alone (cf.\ formula
\eqref{eq:es_gibt_p}), and it appears also in the articles mentioned before.
The remaining condition (A5) just assures that any two adjacent points admit a
description as in this lemma.
\par
In this way we set up a very general framework which can then be applied to
several geometries of matrices. We verify conditions (A1)--(A5) for the
geometry of rectangular matrices over an division ring with more than two
elements, the geometry of Hermitian matrices over a division ring with
involution satisfying some extra conditions, and the projective geometry of
rectangular matrices over an arbitrary division ring. Consequently,
Theorem~\ref{thm:hauptsatz} is applicable to all these geometries. This
improves results from \cite{BlunckHavlicek2005b}, \cite{HavlicekSemrl2006}, and
\cite{Kobal2007} by removing unnecessary assumptions. At the end of
Subsection~\ref{subse:Hermite} we present several examples, for which some of
the conditions (A1)--(A5) are violated. In particular, it is shown that a
diameter preserving surjection need not be an isomorphism for spaces of
symmetric $n\times n$ matrices, $n$ even, over a field of characteristic $2$.

\par We are convinced that there are many more geometries, which allow an
interpretation as a graph with properties (A1)--(A5). Thus, our main result
should also find other applications in the future.
\par
On the other hand, a condition in the spirit of our Lemma~\ref{lem:A1-4} was
also used in situations which are beyond our approach. See
\cite{BlunckHavlicek2005b}, where the points of a graph are defined to be
certain subspaces of a vector space with infinite dimension, and
\cite{HavlicekSemrl2006}, where all bounded linear operators of a complex
Hilbert space with infinite dimension are considered as points of a graph. Any
of the graphs arising in one of these ways has infinite diameter. Nevertheless
it is possible to characterise its adjacency relation in terms of another,
extrinsically given, binary relation. This relation is the complementarity of
two subspaces in \cite{BlunckHavlicek2005b} and the invertibility of the
difference of two operators in \cite{HavlicekSemrl2006}.

\section{The main result}\label{se:hauptsatz}

Let $\Gamma$ be a (finite or infinite) graph. Note that all our graphs are
undirected, without loops and with at least one vertex. The set of vertices of
$\Gamma$ will be denoted by $\cP$. In a more geometric language, vertices will
also be called \emph{points}. As usual, we say that $x,y\in\cP$ are
\emph{adjacent\/} if $\{x,y\}$ is an edge. The \emph{distance\/} of two points
$x,y\in\cP$ is written as $d(x,y)$. Thus $x,y$ are adjacent precisely when
$d(x,y)=1$.
\par
From now on, we focus our attention on graphs $\Gamma$ satisfying the following
conditions:
\begin{enumerate}
\item[(A1)]

$\Gamma$ is connected and its diameter $\diam \Gamma$ is finite.

\item[(A2)]

For any points $x,y\in\cP$ there is a point $z\in\cP$ with
\begin{equation*}
    d(x,z)=d(x,y)+d(y,z)=\diam \Gamma.
\end{equation*}

\item[(A3)]

For any points $x,y,z\in\cP$ with $d(x,z)=d(y,z)=1$ and $d(x,y)=2$ there is a
point $w$ satisfying
\begin{equation*}
    d(x,w)=d(y,w)=1 \mbox{~~and~~} d(z,w)=2.
\end{equation*}

\item[(A4)]
For any points $x,y,z\in\cP$ with $x\neq y$ and $d(x,z)=d(y,z)=\diam \Gamma$
there is a point $w$ with
\begin{equation*}
  d(z,w) = 1,\;\;
  d(x,w) = \diam \Gamma-1, \mbox{~~and~~}
  d(y,w) = \diam \Gamma.
\end{equation*}
\item[(A5)]

For any adjacent points $a,b\in\cP$ there exists a point
$p\in\cP\setminus\{a,b\}$ such that for all $x\in\cP$ the following holds:
\begin{equation*}
    d(x,p)=\diam \Gamma\;\;\Rightarrow\;\;d(x,a)=\diam \Gamma\;\vee\;d(x,b)=\diam \Gamma.
\end{equation*}

\end{enumerate}
Let us shortly comment on these conditions: (A1) is merely a technical
assumption which is needed for all that follows. The subsequent conditions are
about geodesics of $\Gamma$: (A2) says that any geodesic can be extended at
each of its endpoints to a geodesic with length $\diam \Gamma$, which is the
maximal length any geodesic might have. Condition (A3) ensures that for any two
points $x,y$ at distance $2$ there are geodesics $(x,z,y)$ and $(x,w,y)$ with
$d(z,w)=2$. It appears also in \cite{BrouwerWilbrink1983} and
\cite{Cameron1982}. Similarly, (A4) guarantees for distinct points $x,y$ the
existence of a geodesic $(x,\ldots,w,z)$ subject to the specified property of
the penultimate point $w$. Finally, we have our crucial condition (A5): It
states for any two adjacent points the existence of a third point with certain
properties.
\par
We refer to Section~\ref{se:appl} for infinite series of graphs which satisfy
(A1)--(A5). Graphs which satisfy (A1)--(A3), but only one of (A4) and (A5) are
presented in Example~\ref{ex:A4} and Example~\ref{ex:A5}.
\par
Our first result contains a sufficient condition for two points to be adjacent.
Observe that we do not assume condition (A5) here.

\begin{lemma}\label{lem:A1-4}
Given a graph $\Gamma$ which satisfies conditions\/ \emph{(A1)--(A4)} let
$n:=\diam \Gamma$. Suppose that $a,b\in\cP$ are distinct points with the
following property:
\begin{equation}\label{eq:es_gibt_p}
    \exists\, p\in\cP\setminus\{a,b\}\;\forall\, x\in\cP \;:\;
    d(x,p)=n  \;\;\Rightarrow\;\;  d(x,a)=n\;\vee\;d(x,b)=n.
\end{equation}
Then $a$ and $b$ are adjacent.
\end{lemma}

\begin{proof}
Let $k:=d(a,p)$. First we show $k=1$. By condition (A2), there is a point $x\in
\cP$ with
\begin{equation*}
    n=d(x,p)=d(x,a)+d(a,p).
\end{equation*}
Thus $d(x,a)=n-k<n$. We read off from \eqref{eq:es_gibt_p} that $d(x,b)=n$. Now
condition (A4) implies the existence of a point $y\in \cP$ with
\begin{equation*}
    d(x,y)=1,\quad d(y,b)=n-1,\quad d(y,p)=n.
\end{equation*}
So \eqref{eq:es_gibt_p} yields $d(y,a)=n$. Finally,
\begin{equation*}
    n=d(y,a)\le d(y,x)+d(x,a)=1+n-k
\end{equation*}
implies $k=1$, as required. Since property \eqref{eq:es_gibt_p} is symmetric in
$a$ and $b$, we also have $d(b,p)=1$.
\par
Now we prove $d(a,b)=1$. Suppose to the contrary $d(a,b)\neq 1$. From
$d(a,b)\le d(a,p)+d(p,b)=2$, we obtain $d(a,b)=2$. Condition (A3) yields the
existence of a point $w\in\cP$ with
\begin{equation*}
    d(a,w)=d(b,w)=1 \mbox{~~and~~} d(p,w)=2.
\end{equation*}
By (A2), there is a point $z\in \cP$ with
\begin{equation*}
    n=d(z,p)=d(z,w)+d(w,p).
\end{equation*}
Therefore $d(z,w)=n-2$. Furthermore, $d(a,z)\le d(a,w)+d(w,z)=n-1$ and
$d(b,z)\le d(b,w)+d(w,z)=n-1$, a contradiction to property
\eqref{eq:es_gibt_p}.
\end{proof}

We are now in a position to prove our main theorem.

\begin{theorem}\label{thm:hauptsatz}
Let $\Gamma$ and $\Gamma'$ be two graphs satisfying the above conditions\/
\emph{(A1)--(A5)}. If $\phi:\cP\to\cP'$ is a surjection which satisfies
\begin{equation}\label{eq:dm-treu}
    d(x,y)=\diam \Gamma \;\;\Leftrightarrow\;\;
    d(x^{\phi},y^{\phi})=\diam \Gamma'\mbox{~~for all~~} x,y\in\cP,
\end{equation}
then $\phi$ is an isomorphism of graphs. Consequently, $\diam \Gamma=\diam
\Gamma'$.
\end{theorem}

\begin{proof}
We start by showing that $\phi$ is injective. There are two cases as follows.
\par
$\diam \Gamma'=0$: Choose any $x\in\cP$. From $0=d(x^\phi,x^\phi)=\diam
\Gamma'$ follows $0=d(x,x)=\diam \Gamma$. This implies $|\cP|=1$, whence $\phi$
is injective.
\par
$\diam \Gamma'\geq 1$: Let $x, y\in \cP$ be distinct. If $d(x,y)=\diam \Gamma$
then $d(x^\phi,y^\phi)=\diam \Gamma'\geq 1$ so that $x^\phi\neq y^\phi$. Now
suppose that $d(x,y)<\diam \Gamma$. Then, by (A2) and $x\neq y$, there exists a
point $z\in\cP$ for which
\begin{equation*}
    d(x,z) = d(x,y) + d(y,z) =\diam \Gamma \neq d(y,z).
\end{equation*}
Hence $d(x^\phi,z^\phi)=\diam \Gamma'\neq d(y^\phi,z^\phi)$ which shows
$x^\phi\neq y^\phi$.
\par
By the above, we are given a bijection $\phi:\cP\to\cP'$. We infer from
Lemma~\ref{lem:A1-4} and (A5), that $\phi$ preserves adjacency of points in
both directions. Hence it is an isomorphism of graphs.
\end{proof}

\section{Applications}\label{se:appl}

\subsection{Geometry of rectangular matrices}\label{subse:rechteck}
Let $\cD$ be a division ring, $|\cD|\neq 2$, and let $m,n\ge 2$ be integers.
The space of rectangular matrices is based upon set $\MM$ of $m\times n$
matrices with entries in $\cD$. Two matrices $A,B\in \MM$ are defined to be
\emph{adjacent\/} if
\begin{equation*}
    \rank(A-B)=1.
\end{equation*}
Here the term ``rank of a matrix'' is always understood to be the \emph{left
row rank\/}, i.~e., it equals the dimension of the subspace spanned by the row
vectors of the matrix in the left vector space $\cD^n$. It is well known that
the left row rank and the \emph{right column rank\/} coincide for any matrix.
As adjacency is an anti-reflexive and symmetric relation on $\MM$, it can be
viewed as the adjacency relation of a graph with point set $\MM$. It was proved
in \cite[Proposition~3.5]{Wan1996a} that
\begin{equation}\label{eq:d=rank}
    d(A,B) = \rank(A-B) \mbox{~~for all~~}A,B\in \MM.
\end{equation}
\par
We recall that the group $GM_{m\times n}(\cD)$ of transformations
\begin{equation}\label{eq:rechtecktrafo}
    \MM\to \MM : X\mapsto PXQ+R,
\end{equation}
where $P\in\GL_m(\cD)$, $Q\in\GL_n(\cD)$, and $R\in \MM$, is a subgroup of the
automorphism group of the graph on $\MM$.
\par
It was shown in \cite[Corollary~3.10]{Wan1996a} that any two adjacent points
$X,Y\in\MM$ belong to precisely two maximal cliques. Their intersection is
defined to be the \emph{line\/} joining $X$ and $Y$; see
\cite[Corollary~3.13]{Wan1996a}. Moreover, the following holds by
\cite[Lemma~2.2]{HuangWan2004}: Given a point $P\in\MM$ and a line then either
(\emph{i\/}) all points of this line are at the same distance from $P$ or
(\emph{ii\/}) there is an integer $k\geq 1$ such that precisely one point of
this line is at distance $k-1$ from $P$, and all other points of this line are
at distance $k$ from $P$. We shall use this result below.

\begin{lemma}\label{lem:MM_A1-5}
The graph on $\MM$ satisfies conditions\/ \emph{(A1)--(A5)}.
\end{lemma}

\begin{proof}
We denote by $E_{jk}\in \MM$ the matrix whose $(j,k)$ entry equals $1$, whereas
all other entries are $0$. All unordered pairs of matrices with a fixed
distance $k$ are in one orbit under the action of the group $G\MM$. When
exhibiting such a pair we may therefore assume without loss of generality the
two matrices to be $0$ and $E_{11}+E_{22}+\cdots+E_{kk}$.
\par
First, we restrict ourselves to the case $n\geq m$.
\par
Ad (A1): This is immediate from \eqref{eq:d=rank}.
\par
Ad (A2): Let $X=0$ and $Y=\sum_{i=1}^k E_{ii}$. Then $Z:=\sum_{j=1}^n E_{jj}$
has the required properties.
\par
Ad (A3): Let $Y=E_{11}$, $Z=0$, and $X$ be given, where $d(X,Z)=\rank(X)=1$ and
$d(X,Y)=2$. The line joining $Y$ and $Z$ equals $\{ u E_{11} \mid u\in\cD\}$.
The points $Y,Z,-Y$ are on this line. By the preceding remark, all points of
this line, except for $Z$, are at distance $2$ from $X$. In particular,
$d(X,-Y)=2$. Now define $W:=X+Y$. Then $d(W,X)=\rank(Y)=1$, $d(W,Y)=\rank(X)=1$
and $d(W,Z)=\rank(X-(-Y))=d(X,-Y)=2$.
\par
Ad (A4): Let $X\neq Y$ and $Z=0$, whence $\rank(X)=\rank(Y)=m$. With
$x_1,x_2,\ldots,x_m$ and $y_1,y_2,\ldots,y_m$ denoting the row vectors of $X$
and $Y$, respectively, we claim that there exists such an $i\in
\{1,2,\ldots,m\}$ that in $\cD^n$ the $(m-1)$-dimensional affine subspaces
\begin{eqnarray*}
    U_{X,i}&:=&x_i+\spn(x_1,x_2,\ldots,\hat{x}_i,\ldots,x_m),\\
    U_{Y,i}&:=&y_i+\spn(y_1,y_2,\ldots,\hat{y}_i,\ldots,y_m)
\end{eqnarray*}
are distinct. (The notation $\hat{x}_i$ means that this vector is omitted.)
Assume to the contrary that this would not be the case. Then, for any fixed
index $j\in\{1,2,\ldots,m\}$, we would obtain that
\begin{equation*}
    x_j\in\spn(x_1,x_2,\ldots,\hat{x}_k,\ldots,x_m)=\spn(y_1,y_2,\ldots,\hat{y}_k,\ldots,y_m)
\end{equation*}
for all $k\neq j$, whence $x_j\in\spn(y_j)$ due to the linear independence of
the row vectors of $Y$. Furthermore, $U_{X,j}=U_{Y,j}$ would give $x_j=y_j$.
Since $j$ was chosen arbitrarily, we would obtain $X=Y$, a contradiction.
\par
So, we may choose a vector $w\in U_{X,i}\setminus U_{Y,i}$. Define a matrix
$W\in \MM$ as follows: Its $i$th row is equal to $w$, all other rows are $0$.
Then $\rank(W)=1$, $\rank(X-W)=m-1$, and $\rank(Y-W)=m$, as required.
\par
Ad (A5): It suffices to consider the case $A=0$ and $B=E_{11}$. By $|\cD|\ne
2$, the line $\{uE_{11}\mid u\in\cD\}$ contains a point $P\neq A,B$. Let
$X\in\MM$ be any point with $d(X,P)=m$. By the remarks preceding
Lemma~\ref{lem:MM_A1-5} and due to the fact that points with distance $m+1$ do
not exist, at most one of $A$ and $B$ is at distance $m-1$ from $X$.
\par
The case $n\leq m$ can be shown similarly by considering columns of matrices as
vectors of a right vector space over $\cD$.
\end{proof}

By combining Theorem~\ref{thm:hauptsatz} and Lemma~\ref{lem:MM_A1-5} we obtain:

\begin{theorem}\label{thm:rechteck}
Let $\cD,\cD'$ be division rings with $|\cD|,|\cD'|\neq 2$. Let $m,n,p,q$ be
integers $\ge 2$. If $\phi : \MM \to M_{p\times q}(\cD')$ is a surjection which
satisfies
\begin{eqnarray*}
    &\rank(A-B)=\min\{m,n\}  \;\;\Leftrightarrow\;\;
    \rank(A^{\phi}-B^{\phi})=\min\{p,q\}& \\
       &\makebox[0.9\textwidth]{\hfill for all~~$A,B\in \MM$},&
\end{eqnarray*}
then $\phi$ is bijective. Both $\phi$ and $\phi^{-1}$ preserve adjacency of
matrices. Moreover, $\min\{m,n\}=\min\{p,q\}$.
\end{theorem}
The fundamental theorem of the geometry of rectangular matrices
\cite[Theorem~3.4]{Wan1996a} can be used to explicitly describe a mapping
$\phi$ as in the theorem. As a further consequence, the existence of $\phi$
implies that $\cD$ and $\cD'$ are isomorphic or anti-isomorphic division rings,
and that $\{m,n\}=\{p,q\}$.

\subsection{Geometry of Hermitian and symmetric matrices}\label{subse:Hermite}

Let $\cD$ be a division ring which possesses an involution, i.~e.\ an
anti-automorphism of $\cD$ whose square equals the identity map $\id$ of $\cD$.
Throughout this subsection, we choose one involution, say $\inv$, of $\cD$.
Also, we assume that the following restrictions are satisfied:

\begin{itemize}
\item[(R1)]
The set $\cF$ of fixed elements of $\inv$ has more than three elements in
common with the centre $Z(\cD)$ of $\cD$.

\item[(R2)]
When $\inv$ is the identity map, whence $\cD=\cF$ is a field, then assume that
$\cF$ does not have characteristic $2$ (in symbols: $\ch(\cF)\neq 2$).

\end{itemize}
Let $\HH $ denote the space of Hermitian $n\times n$ matrices over $\cD$ (with
respect to $\inv$), where $n\ge 2$. If $\inv$ is the identity map, then $\HH
=:\cS_n(\cF)$ is the space of symmetric $n\times n$ matrices over $\cF$.

We call any Hermitian matrix in $\HH $ a \emph{point\/} and adopt the
\emph{adjacency relation\/} from \ref{subse:rechteck}, i.~e., $A,B\in\HH $ are
adjacent precisely when $\rank(A-B)=1$. This turns $\HH $ into a graph. We
recall that the group $G\HH $ of transformations
\begin{equation}\label{eq:herm-trafo}
    \HH \to\HH  : X \mapsto P X\overline{P}{}^t + H,
\end{equation}
where $P\in\GL_n(\cD)$ and $H\in \HH $, is a subgroup of the automorphism group
of the graph on $\HH $.
\par
For any two matrices $A,B\in \HH $ the distance $d(A,B)$ in the graph on $\HH$
equals $\rank(A-B)$. This can be shown, mutatis mutandis, as in
\cite[Proposition~5.5]{Wan1996a}, because (R2) guarantees that any Hermitian
matrix is cogredient to a matrix of the form $\sum_{i=1}^n a_i E_{ii}$ with
$a_i\in\cF$. See, for example, \cite[p.~15]{Dieudonne1971}.
\par

\begin{lemma}\label{lem:Lemma2.2}
Let $A\in\HH $ be a matrix with $\rank(A)=k+1$. A matrix $B\in\HH $ has rank
$1$ and $\rank(A-B)=k$ if, and only if, there exists an $x\in \cD^n$ with
$xA\overline{x}{}^t\neq 0$ and
\begin{equation*}
    B=\overline{(xA)}{}^t  (xA\,\overline{x}^t){}^{-1}   (xA).\end{equation*}
\end{lemma}
\begin{proof}
This is a slight generalisation of Lemma 2.2 in \cite{HuangHoeferWan2004},
since we do \emph{not\/} assume $\cF\subseteq Z(\cD)$. However, the proof given
there can be carried over to our more general settings in a straightforward
way. On the one hand, all scalars in $\cF$ (like $(xA\,\overline{x}^t){}^{-1}$
in the definition of $B$ from above) have to be written \emph{between\/} a
matrix and its Hermitian transpose rather than on the left hand side (as in
\cite{HuangHoeferWan2004}). Also, one has to take into account what we already
noticed before: In the presence of restriction (R2), any Hermitian matrix is
cogredient to a diagonal matrix (with entries in $\cF$) irrespective of whether
$\cF$ is in the centre of $\cD$ or not.
\end{proof}

\begin{lemma}\label{lem:x.A.B}
Let $A,B\in\HH $ be non-zero, and suppose that there exists $P\in\GL_n(\cD)$
such that
\begin{equation*}
    PA\overline{P}{}^t = \diag(a_1,a_2,\ldots,a_k,0,\ldots,0)
    \mbox{~~and~~}
    PB\overline{P}{}^t =\begin{pmatrix}B_1 & 0\\0&0\end{pmatrix},
\end{equation*}
where $k=\rank(A)$ and $B_1$ denotes a Hermitian matrix of size $\leq k$. Then
there is a vector $x\in \cD^n$ such that
\begin{equation*}
    x A \overline{x}{}^t\neq 0 \mbox{~~and~~}
    x B \overline{x}{}^t\neq 0.
\end{equation*}
\end{lemma}

\begin{proof}
Without loss of generality, let $k=n$, whence $A=\diag(a_1,a_2,\ldots,a_n)$,
$a_i\in \cF\setminus\{0\}$, and $B=(b_{ij})\neq 0$.
\par
\emph{Case 1.} $b_{ii}\neq 0$ for some $i$. Then $e_i$, viz.\ the $i$th vector
of the canonical basis of $\cD^n$, satisfies
\begin{equation*}
    e_i A \overline{e}_i{}^t=a_i\neq 0
    \mbox{~~and~~}
    e_i B \overline{e}_i{}^t=b_{ii}\neq 0.
\end{equation*}
\par
\emph{Case 2.} $b_{ii}=0$ for all $i$. Since $B\neq 0$, there exist $i,j$ with
$1\le i,j\le n$ and $i\neq j$ such that $b_{ij}\neq 0$. Without loss of
generality, we assume $b_{12}\neq 0$. Let $x=(x_1,1,0,\ldots,0)$, then $x
A\overline{x}{}^t= x_1 a_1 \overline{x}_1+a_2$ and $xB\overline{x}{}^t=x_1
b_{12}+\overline{x_1 b_{12}}$, so it is enough to find $x_1\in\cD$ such that
\begin{equation*}
    x_1 a_1\overline{x}_1\neq - a_2
    \mbox{~~and~~}
    x_1 b_{12} \neq -\overline{x_1 b_{12}}.
\end{equation*}

As $|\cF\cap Z(\cD)|> 3$, there exists $\lambda\in \big(\cF\cap Z(\cD)\big)
\setminus \{ 0 \} $ with $\lambda^2\neq 1$. Note that
$\cD=\{\xi\in\cD\mid\xi=-\overline{\xi}\}$ would imply $(\inv)=\id$ and
$\ch\cD=2$, which contradicts (R2). So, there is $x_1'\in\cD$ with $x_1' b_{12}
\neq -\overline{x_1' b_{12}}$. Define $x_1:=x_1'$ if $x_1 a_1 x_1'\neq -a_2$,
and $x_1:=\lambda x_1'$ if $x_1 a_1 x_1' = -a_2$.
\end{proof}

\begin{lemma}\label{lem:HH_A1-5}
The graph on $\HH $ satisfies conditions\/ \emph{(A1)--(A5)}.
\end{lemma}

\begin{proof}
When exhibiting two Hermitian matrices with distance $k$, we may assume, by
virtue of the action of $G\HH$, the matrices to be $0$ and
$a_1E_{11}+a_2E_{22}+\cdots+a_kE_{kk}$ with $a_1,a_2,\ldots,a_k\in\cF$. Taking
into account the previous remark, the proof for (A1), (A2), (A3), and (A5) can
be carried over almost unchanged from the proof of Lemma~\ref{lem:MM_A1-5}.
Only certain scalars have to chosen from $\cF\cap Z(\cD)$ rather than $\cD$.
\par
Our proof of (A4) is different though: Let $X\neq Y$ and $Z$ be matrices in
$\HH $ with $d(X,Z)=n$ and $d(Y,Z)=n$. Without loss of generality, we assume
$Z=0$ and $\rank(X)=\rank(Y)=n$. From Lemma~\ref{lem:x.A.B}, applied to $A:=X$
and $B:=X-XY^{-1}X\neq 0$, there exists a vector $v\in \cD^n$ such that
\begin{equation*}
    v X \overline{v}{}^t\neq 0
    \mbox{~~and~~}
    v X\overline{v}{}^t-v(XY^{-1}X)\overline{v}{}^t\neq 0.
\end{equation*}
We define
\begin{equation}\label{eq:Wv}
    W:=(\overline{vX})^t \, (vX\overline{v}{}^t)^{-1} \, (vX) \in\HH .
\end{equation}
Then $d(Z,W)=1$ and $d(Y,W)\ge n-1$ are obvious, whereas Lemma~\ref{lem:x.A.B}
shows $d(X,W)=n-1$. Let us suppose $d(Y,W)=n-1$. By Lemma~\ref{lem:Lemma2.2},
there exists a vector $u\in \cD^n$ such that
\begin{equation}\label{eq:Wu}
    W=(\overline{uY}){}^t  \,  (uY\overline{u}{}^t)^{-1}  \,   (uY).
\end{equation}
We infer from \eqref{eq:Wu} and \eqref{eq:Wv} that $uY$ and $vX$ are
left-proportional by a non-zero factor in $\cD$. Since $u$ is determined up to
a non-zero factor in $\cD$ only, we may therefore even suppose $uY=vX$.
Comparing \eqref{eq:Wu} with \eqref{eq:Wv} yields
$vX\overline{v}^t=uY\overline{u}^t$. This implies that
$vX\,\overline{v}^t-v(XY^{-1}X)\,\overline{v}^t=0$, a contradiction. So we must
have $d(Y,W)=n$.
\end{proof}

\begin{theorem}\label{thm:Hermite}
Let $\cD, \cD'$ be division rings which possess involutions $\inv$ and
$\invstrich$, respectively, subject to the restrictions\/ \emph{(R1)} and\/
\emph{(R2)}. Let $n,n'$ be integers $\ge 2$. If $\phi: \HH \to \cH_{n'}(\cD')$
is a surjection which satisfies
\begin{equation*}
    \rank(A-B)=n \;\;\Leftrightarrow\;\; \rank(A^{\phi}-B^{\phi})=n'
    \mbox{~~for all~~} A,B\in \cH_n(\cD),
\end{equation*}
then $\phi$ is bijective. Both $\phi$ and $\phi^{-1}$ preserve adjacency of
Hermitian matrices. Moreover, $n=n'$.
\end{theorem}
A prospective \emph{fundamental theorem of the geometry of Hermitian
matrices\/} should describe all bijections $\HH\to \cH_{n'}(\cD')$ which
preserve adjacency in both directions. However, such a fundamental theorem
seems to be known only under additional assumptions on the division rings,
their involutions, and/or the numbers $n,n'$. We refer to \cite{Huang2006x},
\cite{WanHuang2002}, \cite{WanHuang2006}, and \cite[Chapter~6]{Wan1996a} for
further details. Each of these results can be used to (\emph{i\/}) explicitly
describe a mapping $\phi$ as in the theorem and (\emph{ii\/}) to derive from
the existence of $\phi$ that $\cD$ and $\cD'$ are isomorphic division rings.
\par
We close this subsection with some examples in which one or even both of the
restrictions (R1) and (R2) dropped.

\begin{example}\label{ex:A4}
Let $\cF_3$ be the field with three elements. We exhibit the space of symmetric
$2\times 2$ matrices over $\cF_3$. The graph on $\cS_2(\cF_3)$ has $27$ points
and diameter $2$. It is easy to verify conditions (A1), (A2), (A3), and (A5) as
before.
\par
In what follows we establish that (A4) is not satisfied. Figure~\ref{fig:A4}
depicts five points of the graph on $\cS_2(\cF_3)$ and all edges between them.
\begin{figure}[ht!]\footnotesize\unitlength5pt
\centering
\begin{picture}(48.2,9.4)\thicklines
    \put(0,1.4){$Y=\begin{pmatrix}2&2\\2&1\end{pmatrix}$}
    \put(0,6.6){$X=\begin{pmatrix}1&0\\0&2\end{pmatrix}$}

    \put(10.5,1.8){\line(3,2){7.8}}
    \put(10.5,7.0){\line(3,-2){7.8}}
    \put(10.5,1.8){\line(1,0){7.8}}
    \put(10.5,7.0){\line(1,0){7.8}}

    \put(19.2,1.4){$V=\begin{pmatrix}0&0\\0&2\end{pmatrix}$}
    \put(19.2,6.6){$U=\begin{pmatrix}1&0\\0&0\end{pmatrix}$}

    \put(38.5,4.0){$\begin{pmatrix}0&0\\0&0\end{pmatrix}=Z$}

    \put(29.8,1.8){\line(3,1){7.8}}
    \put(29.8,7.0){\line(3,-1){7.8}}
\end{picture}
\caption{A counterexample for (A4) and Lemma~\ref{lem:A1-4}}\label{fig:A4}
\end{figure}
It is straightforward to show that $(X,U,Z)$ and ($X,V,Z)$ are the only two
geodesics from $X$ to $Z$. However, both $U$ and $V$ are neighbours of $Y\neq
X$, whence we cannot find a matrix $W$ to satisfy (A4).
\par
Furthermore, property (\ref{eq:es_gibt_p}) holds for $A:=X$, $B:=Z$, and
$P:=Y$. Indeed, $U$ and $V$ are the only points of $\cS_2(\cF_3)$ which are
adjacent to $A$ \emph{and\/} $B$, but none of them is at distance $2$ from $P$.
Yet, in contrast to the assertion of Lemma~\ref{lem:A1-4}, the points $A$ and
$B$ are not adjacent.
\par
Nevertheless, any mapping $\phi:\cS_2(\cF_3)\to\cS_2(\cF_3)$ as in
Theorem~\ref{thm:Hermite} is an automorphism of the graph on $\cS_2(\cF_3)$, a
fact which is immediate from the following observation: Given a mapping
$\phi:\cP\to\cP$ as in Theorem~\ref{thm:hauptsatz}, where $\Gamma=\Gamma'$ is a
finite graph with diameter $\diam\Gamma=2$, the surjectivity of $\phi$ implies
its being a bijection. Furthermore, since distance $2$ is preserved under
$\phi$ and $\phi^{-1}$, so is distance $1$. Hence $\phi$ is an automorphism.
\end{example}

\begin{example}\label{ex:A5}
Let $\cF_2$ be the field with two elements. We exhibit the space of symmetric
$2\times 2$ matrices over $\cF_2$. The graph on $\cS_2(\cF_2)$ has $8$ points
and diameter $3$, an illustration is given in Figure~\ref{fig:A5}. It is
straightforward to show that conditions (A1), (A2), (A3), and (A4) are
satisfied, whereas (A5) does not hold.
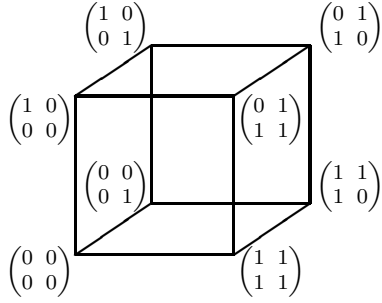
\begin{figure}[ht!]\footnotesize\unitlength12pt
\centering
\begin{picture}(11.6,9.5)(-1.5,-1)\thicklines
        \put(0.4,0.4){\line(1,0){5}}
        \put(0.4,5.4){\line(1,0){5}}
        \put(0.4,0.4){\line(0,1){5}}
        \put(5.4,0.4){\line(0,1){5}}
        \put(2.8,2.0){\line(1,0){5}}
        \put(2.8,7.0){\line(1,0){5}}
        \put(2.8,2.0){\line(0,1){5}}
        \put(7.8,2.0){\line(0,1){5}}
        \put(0.4,0.4){\line(3,2){2.4}}
        \put(0.4,5.4){\line(3,2){2.4}}
        \put(5.4,0.4){\line(3,2){2.4}}
        \put(5.4,5.4){\line(3,2){2.4}}
    \put(-1.8,-0.3){$\begin{pmatrix}0\;\;0\\0\;\;0\end{pmatrix}$}
    \put(-1.8,4.5){$\begin{pmatrix}1\;\;0\\0\;\;0\end{pmatrix}$}
    \put(0.6,2.4){$\begin{pmatrix}0\;\;0\\0\;\;1\end{pmatrix}$}
    \put(0.6,7.4){$\begin{pmatrix}1\;\;0\\0\;\;1\end{pmatrix}$}
    \put(5.5,-0.3){$\begin{pmatrix}1\;\;1\\1\;\;1\end{pmatrix}$}
    \put(8,2.4){$\begin{pmatrix}1\;\;1\\1\;\;0\end{pmatrix}$}
    \put(8,7.4){$\begin{pmatrix}0\;\;1\\1\;\;0\end{pmatrix}$}
    \put(5.5,4.5){$\begin{pmatrix}0\;\;1\\1\;\;1\end{pmatrix}$}
\end{picture}
\caption{A counterexample for (A5) }\label{fig:A5}
\end{figure}
\par
Another way of seeing that the graph on $\cS_2(\cF_2)$ cannot satisfy all
conditions (A1)--(A5) is as follows. Suppose that $\Gamma$ is a graph with
diameter $\diam\Gamma\geq 3$ such that there exist points $a,a^*\in\cP$ with
 $d(a,a^*)=\diam\Gamma$ and $d(a,x)\neq\diam\Gamma\neq d(a^*,x)$ for all
$x\in\cP\setminus\{a,a^*\}$. Let $\phi:\cP\to\cP$ be the bijection which
interchanges $a$ with $a^*$ and leaves invariant all other points. This $\phi$
preserves pairs of points with distance $\diam\Gamma$ in both directions. But,
due to $\diam\Gamma\geq 3$, the bijection $\phi$ cannot be an automorphism of
$\Gamma$. Clearly, the graph on $\cS_2(\cF_2)$ is of this kind.
\end{example}

\begin{example}\label{ex:Sym-Alt}
The space $\cS_2(\cF_2)$ from Example~\ref{ex:A5} is just a particular case of
the following, more general situation. Let $\cF$ be any field of characteristic
$2$, and let $n\geq 2$ be an even integer. By \cite[Proposition~5.5]{Wan1996},
the diameter of the graph on the space $\cS_n(\cF)$ equals $n+1\ge 3$.
Moreover, two matrices $A,B\in\cS_n(\cF)$ satisfy $d(A,B)=n+1$ if, and only if,
$A-B$ is an alternate matrix with rank $n$. Consequently, $d(A,B)= n+1$ implies
that either \emph{both $A$ and $B$ are alternate\/} or \emph{both $A$ and $B$
are non-alternate}. Now it is easy to establish the existence of a bijection
$\phi:\cS_n(\cF)\to\cS_n(\cF)$ which preserves pairs of matrices at distance
$n+1$ in both directions without being an isomorphism. Choose any alternate
matrix $K\in\cS_n(\cF)$ with $K\neq 0$. Given $X\in\cS_n(\cF)$ we define
\begin{equation*}
    X^\phi:=X+K   \mbox{~~if~~} X \mbox{~~is alternate, and~~}
    X^\phi:=X     \mbox{~~otherwise}.
\end{equation*}
As the restriction of $\phi$ to the set of alternate matrices is a
transformation as in \eqref{eq:herm-trafo}, $\phi$ preserves matrix pairs with
distance $n+1$. We have $d(E_{11},0)=1$ and
\begin{equation*}
          d(0,E_{11}) + d(E_{11},K) \ge d(0,K) = \rank(K)+1\ge 3.
\end{equation*}
Hence $d(E_{11}^\phi,0^\phi)=d(E_{11},K)\ge 2$.
\end{example}

\subsection{Projective geometry of rectangular
matrices---\newline the Grassmann space}

Let $\cD$ be a division ring. The projective space of rectangular matrices
$\MM$, $m,n\ge 2$, is the Grassmann space $G(m,m+n; \cD)$ over $\cD$; its
\emph{points\/} are the $m$-dimensional subspaces of the $(m+n)$-dimensional
left vector space over $\cD$. We refer to \cite[Section~3.6]{Wan1996a} for its
relationship with $\MM$. Two points $W_1$,$W_2\in G(m,m+n; \cD)$ are called
\emph{adjacent\/} if $W_1\cap W_2$ is $(m-1)$-dimensional. As before, we
consider $G(m,m+n; \cD)$ as a graph based on the adjacency relation. The
distance between two points $W_1$ and $W_2$ is
\begin{equation*}
    d(W_1,W_2)=m-\dim(W_1\cap W_2).
\end{equation*}
The graph on the Grassmann space $G(m,m+n; \cD)$ has diameter $\min\{m,n\}$.

Using dimension arguments, conditions (A1), (A2), (A3), and (A5) can be proved
easily. We sketch the proof of (A4) for the case $m\leq n$. Given
$m$-dimensional subspaces $X,Y,Z$ with $X\neq Y$ and $d(X,Z)=d(Y,Z)=m$ there
exists a vector $a\in X\setminus Y$. Choose an $(m-1)$-dimensional subspace
$S\subset Z$ such that $S\cap\big({\spn(a,Y)}\cap Z\big)=\{0\}$. Then
$W:=\spn(a,S)$ has the required properties.

Due to the presence of \emph{points at infinity\/} there is no need to exclude
the field with two elements in the following theorem.

\begin{theorem}\label{thm:grassmann}
Let $\cD,\cD'$ be division rings. Let $m,n,p,q$ be integers $\ge 2$. If $\phi:
G(m,m+n; \cD)\to G(p,p+q;\cD' )$ is a surjection which satisfies
\begin{eqnarray*}
    &d(A,B)=\min\{m,n\} \;\;\Leftrightarrow\;\;
     d(A^{\phi},B^{\phi})=\min\{p,q\}& \\
       &\makebox[0.9\textwidth]{\hfill for all~~$A,B\in G(m,m+n; \cD)$},&
\end{eqnarray*}
then $\phi$ is bijective. Both $\phi$ and $\phi^{-1}$ preserve adjacency of
subspaces. Moreover, $\min\{m,n\}=\min\{p,q\}$.
\end{theorem}
The fundamental theorem of the projective geometry of rectangular matrices
\cite[Theorem~3.52]{Wan1996a} can be used to explicitly describe a mapping
$\phi$ as in the theorem. As a further consequence, the existence of $\phi$
implies that $\cD$ and $\cD'$ are isomorphic or anti-isomorphic division rings,
and that $\{m,n\}=\{p,q\}$.

\bibliographystyle{plain}
\bibliography{mrabbrev,schaum}

Wen-ling Huang, Institut f\"{u}r Diskrete Mathematik und Geometrie, Tech\-ni\-sche
Universit\"{a}t Wien, Wiedner Hauptstra{\ss}e 8--10, A-1040 Wien, Austria.\\
Department Mathematik, Universit\"{a}t Hamburg, Bundesstra{\ss}e 55, D-20146 Hamburg,
Germany.\\
\texttt{huang@math.uni-hamburg.de}\\

Hans Havlicek, Institut f\"{u}r Diskrete Mathematik und Geometrie, Tech\-ni\-sche
Universit\"{a}t Wien, Wiedner Hauptstra{\ss}e 8--10, A-1040 Wien, Austria.\\
\texttt{havlicek@geometrie.tuwien.ac.at}

\end{document}